\documentclass[11pt,oneside,reqno]{amsart}

\usepackage[T1]{fontenc}
\usepackage{amssymb}
\usepackage{algpseudocode}
\usepackage{algorithm}
\usepackage{verbatim}
\usepackage{graphicx}
\usepackage{placeins}
\usepackage{listings}
\usepackage{xcolor}
\usepackage{enumerate}
\usepackage{subcaption}
\usepackage[noadjust]{cite}
\RequirePackage{dsfont} 
 \scrollmode
\allowdisplaybreaks
\usepackage{graphicx}
\usepackage{epstopdf}
\usepackage{hyperref} 
\usepackage{mathtools}
\usepackage{float}

\usepackage{xcolor}

\usepackage{amsmath}
\usepackage{tikz}
\usepackage{enumerate}
\usepackage{xcolor}
\usepackage{enumitem}


\definecolor{codegreen}{rgb}{0,0.6,0}
\definecolor{codegray}{rgb}{0.5,0.5,0.5}
\definecolor{codepurple}{rgb}{0.58,0,0.82}
\definecolor{backcolour}{rgb}{0.95,0.95,0.92}

\lstdefinestyle{list_style}{
  backgroundcolor=\color{backcolour}, commentstyle=\color{codegreen},
  keywordstyle=\color{magenta},
  numberstyle=\tiny\color{codegray},
  stringstyle=\color{codepurple},
  basicstyle=\ttfamily\footnotesize,
  breakatwhitespace=false,         
  breaklines=true,                 
  captionpos=b,                    
  keepspaces=true,                 
  numbers=left,                    
  numbersep=5pt,                  
  showspaces=false,                
  showstringspaces=false,
  showtabs=false,                  
  tabsize=2
}

\newcommand{\xdasharrow}[2][->]{
\tikz[baseline=-\the\dimexpr\fontdimen22\textfont2\relax]{
\node[anchor=south,font=\scriptsize, inner ysep=1.5pt,outer xsep=2.2pt](x){#2};
\draw[shorten <=3.4pt,shorten >=3.4pt,dashed,#1](x.south west)--(x.south east);
}
}

\newcommand{\DEBUG}{}

\ifdefined\DEBUG%

  \def\rem#1{{\marginpar{\raggedright\scriptsize #1}}}
  \newcommand{\pmr}[1]{\rem{\color{blue}{$\bullet$ #1}}}
  \newcommand{\ppr}[1]{\rem{\color{red}{$\bullet$ #1}}}
 \else%

  \newcommand{\ppr}[1]{}
  \newcommand{\pmr}[1]{}
 \fi

\def\rho{\varrho_1}

\theoremstyle{plain}
\newtheorem{theorem}{Theorem}
\newtheorem{lemma}{Lemma}
\newtheorem{fact}{Fact}

\theoremstyle{definition}
\newtheorem{remark}{Remark}

\theoremstyle{definition}
\newtheorem{assumption}[theorem]{Assumption}

\begin{document}

\title[On the error of the Euler scheme for DDEs under inexact information]{On the error of the Euler scheme for approximation of solutions of nonlinear DDEs under inexact information}

\author[P. Przyby\l owicz]{Pawe{\l } Przyby\l owicz}
 \address{AGH University of Krakow,
Faculty of Applied Mathematics,
 Al. A.~Mickiewicza 30, 30-059 Krak\'ow, Poland}
 \email{pprzybyl@agh.edu.pl}

\author[M. Wi\c{a}cek]{Martyna Wi\c{a}cek}
 \address{AGH University of Krakow,
Faculty of Applied Mathematics,
 Al. A.~Mickiewicza 30, 30-059 Krak\'ow, Poland}
\email{martynawiacek@agh.edu.pl, corresponding author}

\begin{abstract}

We analyze the behavior of the Euler method for delay differential equations under nonstandard assumptions on the right-hand-side function $f$, when evaluations of $f$ are corrupted by information noise. We consider a deterministic inexact-information model in which perturbations affect the numerical evaluation of the right-hand side and may propagate across successive delay intervals.

We provide theoretical upper bounds on the Euler discretization error in two settings: first, under global Lipschitz assumptions, and second, under a local one-sided Lipschitz condition combined with local H\"older continuity. In the globally Lipschitz case we obtain stability with respect to perturbations and derive a global error estimate in terms of the time step $h$ and the noise level $\delta$. In the weaker regularity regime, we show that the interaction between the delay term and information noise leads to a more delicate error structure, including a hierarchy of exponents depending on the H\"older parameter of the delayed argument. This turned out to be different for DDEs, compared to ODEs.

We also present numerical experiments illustrating the convergence behavior of the noisy Euler scheme and confirming the theoretical estimates. In particular, the experiments show how the accumulation of perturbations becomes more pronounced when the regularity in the delayed variable is weaker.
\end{abstract}
\maketitle
\tableofcontents
\section{Introduction}
In this paper, we study the approximation of solutions to delay differential equations (DDEs) with information noise using the Euler scheme.

Let us consider the problem of approximating the solution 
$z : [0, +\infty) \to \mathbb{R}^d$ of the multidimensional delay differential equation of the form
\begin{equation}\label{main_eq}
\begin{cases}
z'(t) = f(t, z(t), z(t - \tau)), & t \in [0, (n + 1)\tau], \\
z(t) = \eta, & t \in [-\tau, 0],
\end{cases}
\end{equation}
with a constant time lag $\tau \in (0, +\infty)$. Here $\eta \in \mathbb{R}^d$ is the initial condition, $n \in \mathbb{N}$ is a (finite and fixed) \emph{horizon parameter}, and the right-hand side function 
$f : [0, +\infty) \times \mathbb{R}^d \times \mathbb{R}^d \to \mathbb{R}^d$ 
for a fixed $d \in \mathbb{N}$ satisfies the  appropriate regularity conditions. In this work, we investigate the approximation error of the Euler scheme under two distinct assumptions on the right-hand side: first, a global Lipschitz condition; second, a local one-sided Lipschitz with Hölder regularity, which is a nonstandard assumption.

Delay differential equations (DDEs) are a well-established class of functional differential equations used to model systems in which the present rate of change depends not only on the current state but also on past states. Classical monographs such as \cite{bellen_zennaro}, \cite{hale_verduyn_lunel} provide a comprehensive theory of existence, uniqueness, and numerical approximation under global Lipschitz assumptions. These conditions ensure good stability properties of numerical schemes, including the Euler method, but they are often too restrictive for nonlinear or application-driven models.

In recent years, attention has shifted toward studying DDEs under weaker and more realistic regularity assumptions. In particular, an approximation theory for the Euler method in the setting of locally one-sided Lipschitz and H\"older continuous $f$'s has been developed in \cite{ncz_pp_pm}, where convergence and optimal-order error estimates were established for the noise-free case. Even more general Carath\'eodory-type right-hand sides, measurable in time and only continuous in the state variables, have been investigated in \cite{difonzo_przybylowicz_wu_2024}, where the existence of generalized solutions and stability under perturbations were analyzed.

On the other hand, the literature on perturbed or inexact information for DDEs is still very limited. While stochastic delay differential equations (SDDEs) with random noise in the dynamics have been extensively studied (see, e.g., \cite{buckwar_sdde_survey, kumar_sabanis_sdde, przybylowicz_wu_xie_2024}), much less is known about deterministic information noise affecting the numerical evaluation of the right-hand side function. Such perturbations naturally arise from finite precision, rounding and discretization errors, or inexact preprocessing steps in the computational pipeline. To the best of our knowledge, a systematic error analysis of the Euler scheme for DDEs under deterministic information noise (especially in the weak regularity regimes considered in this paper) has not yet been developed.

A closely related line of research has been developed for problems without delay, where the availability of only \emph{inexact (noisy) standard information} about the data is a natural modeling paradigm in information-based complexity. In the context of initial value problems for ODEs, robust algorithms and optimal error bounds under noisy evaluations of the right-hand side were studied, among others, in \cite{kacewicz_przybylowicz_2016}. More recent works analyze explicit/implicit \emph{randomized} Euler-type schemes for ODEs when only perturbed function values are accessible and investigate stability and optimality in such noise models, see, e.g., \cite{bochacik_przybylowicz_2022} as well as higher-order randomized Runge--Kutta constructions in \cite{bochacik_gocwin_morkisz_przybylowicz_2021}. In the SDE setting, inexact information models have been used to quantify how deterministic perturbations in the evaluation of the drift and diffusion coefficients, as well as perturbations of the driving Wiener process, affect strong approximation, in particular, optimal pointwise approximation and randomized Euler-type methods under noisy information were studied in \cite{morkisz_przybylowicz_2017}, and further developments include randomized Euler algorithms for SDEs with disturbed information, see, e.g., \cite{baranek_kaluza_morkisz_przybylowicz_sobieraj_2023}.  Finally, related issues arise already at the level of numerical integration: approximation of stochastic integrals under analytic noise models (including explicit links to low-precision computation) was investigated in \cite{kaluza_morkisz_przybylowicz_2019}.

From the practical perspective, studying inexact information is motivated by the fact that modern simulation pipelines rarely evaluate $f$ exactly. Perturbations may arise from finite-precision arithmetic, such as rounding and loss of significant digits, from implementations that use mixed or low numerical precision, for example to improve performance on GPUs or other accelerators, from approximate evaluations of transcendental functions, from surrogate or interpolatory approximations of coefficients, as well as from inexact preprocessing and data movement.
A deterministic bounded-noise model offers a tractable way to capture these effects and to understand the robustness of time-stepping schemes with respect to implementation-level inaccuracies. In DDEs this issue is further amplified by the delayed argument, as perturbations can propagate across consecutive delay intervals.

\vspace{2mm}

The main contributions of the paper are as follows.
\begin{enumerate}[label=(\roman*)]
\item We extend the error analysis of the Euler scheme for nonlinear delay differential equations to the setting of inexact information about the right-hand side. In particular, we derive stability estimates and global error bounds that quantify the dependence of the Euler error on the time step $h$, the noise level $\delta$, and the regularity parameters of $f$ under both global Lipschitz and one-sided Lipschitz/Hölder assumptions.
\item We show that under one-sided Lipschitz and Hölder conditions, the deterministic information noise may accumulate along the delay intervals, which leads to a characteristic hierarchy of error exponents in $h$ and $\delta$.
\item We present numerical experiments that confirm the theoretical convergence rates and illustrate the effect of noise accumulation.
\end{enumerate}
The paper is organized as follows. Section~2 introduces the inexact information model and the noisy Euler scheme. Section~\ref{sec:lipschitz} is devoted to the error analysis under global Lipschitz assumptions. Section~\ref{sec:onesidedlipschitz} contains the corresponding analysis in the one-sided Lipschitz and locally Hölder continuous case. Section~5 reports results of numerical experiments. Section~6 collects concluding remarks and outlines possible directions for future research. An Appendix gathers auxiliary analytical results and discrete Gronwall-type inequalities used in the proofs.



\section{Inexact information model of computation}

For $x, y \in \mathbb{R}^d$ we take 
$
\langle x, y \rangle = \sum_{k=1}^{d} x_k y_k $, $\|x\| = \langle x, x \rangle^{1/2}.$

\begin{assumption}[Noise condition]\label{ass:noise}
The perturbed (noisy) model is given by
\begin{equation}\label{noise}
\tilde{f}(t,y,z) = f(t,y,z) + \tilde{\delta}_f(t,y,z),
\end{equation}
where $\tilde{\delta}_f : [0,+\infty)\times\mathbb{R}^d\times\mathbb{R}^d\to\mathbb{R}^d$
is a Borel measurable function satisfying
\[
\| \tilde{\delta}_f(t,y,z) \| 
\le \delta (1 + \|y\|)(1 + \|z\|),
\qquad \delta \in [0,1].
\]
\end{assumption}
The multiplicative form of the perturbation is natural in information-based error models. Namely, it allows bounded relative inaccuracy that scales with the size of the arguments and preserves the polynomial growth structure of the perturbed $f$.\\

Having introduced the noise model, we now proceed to formulate the Euler approximation that incorporates such deterministic perturbations.

\vspace{0.5cm}


Let $N \in \mathbb{N}$ be fixed, and define the discretization parameters as
\[
h = \frac{\tau}{N}, 
\qquad 
t_k^j = j\tau + k h, 
\quad 
k = 0,1,\dots, N, \quad 
j = 0,1,\dots, n.
\]

The initial condition is prescribed as
\[
\tilde{y}_0^{-1} = \tilde{y}_1^{-1} = \dots = \tilde{y}_N^{-1} = \tilde{\eta},
\qquad 
\| \tilde{\eta} - \eta \| \le \delta.
\]

Then, for $j = 0,1,\dots, n$ and $k = 0,1,\dots, N-1$, the Euler scheme is defined recursively by
\[
\begin{cases}
\tilde{y}_0^j = \tilde{y}_N^{j-1}, \\[4pt]
\tilde{y}_{k+1}^j = \tilde{y}_k^{j} 
+ h\, \tilde{f}\!\left( t_k^j, \tilde{y}_k^{j}, \tilde{y}_k^{j-1} \right).
\end{cases}
\]

In what follows, we denote by $\{y_k^j\}$ the Euler values computed with the exact right-hand side $f$, while $\{\tilde y_k^j\}$ denote the Euler values computed with the perturbed right-hand side $\tilde f$ (inexact information).
The sequence $\{y_k^j\}$ is generated by the same recursion with $f$ in place of $\tilde f$ and $\eta$ in place of $\tilde\eta$.

For each $j=0,\dots,n$, let $l_N^j$ and $l_{N,\delta}^j$ denote the piecewise-linear interpolants of $\{y_k^j\}_{k=0}^N$ and $\{\tilde y_k^j\}_{k=0}^N$ on $[j\tau,(j+1)\tau]$.
We also define the global interpolants
\[
l_N(t):=l_N^j(t),\qquad l_{N,\delta}(t):=l_{N,\delta}^j(t),
\quad t\in[j\tau,(j+1)\tau],\ \ j=0,\dots,n.
\]

\noindent
The objective of this paper is to analyse the accuracy of the Euler method applied to delay differential equations with inexact information on the right-hand side. In particular, we aim to derive upper bounds for the global approximation error
\[
\max_{0 \le j \le n}\; \max_{0 \le k \le N} \bigl\| y_k^{\,j} - \tilde{y}_k^{\,j} \bigr\|,
\]
under various structural assumptions imposed on the right-hand side function~$f$ and

\[
\sup_{t \in [0,(n+1)\tau]}\;  \bigl\| z(t) - l_{N,\delta}(t)\bigr\|.
\]

Each chapter considers a different set of assumptions on~$f$, leading to corresponding error estimates and illustrating how these assumptions influence the behaviour of the Euler scheme.

\section{Analysis under the Global Lipschitz Condition}
\label{sec:lipschitz}

\subsection{Problem formulation}

We assume that $f$ in (\ref{main_eq}) satisfies
\begin{itemize}
    \item[(E1)] $f \in C([0,+\infty)\times\mathbb{R}^d\times\mathbb{R}^d;\mathbb{R}^d)$,
    \item[(E2)] there exists $L\ge 0$ such that for all $t_1,t_2\in[0,(n+1)\tau]$ and $y_1,y_2,z_1,z_2\in\mathbb{R}^d$,
    \[
      \|f(t_1,y_1,z_1)-f(t_2,y_2,z_2)\|
      \le L\big(|t_1-t_2|+\|y_1-y_2\|+\|z_1-z_2\|\big).
    \]
\end{itemize}

The condition \textnormal{(E2)} means that the function $f$ satisfies the global Lipschitz condition 
with respect to all its variables with a constant $L$. 
Note that assumptions \textnormal{(E1)}–\textnormal{(E2)} guarantee the existence and uniqueness 
of the solution $z = z(t)$ on the entire interval 
$t \in [-\tau, (n+1)\tau]$ for problem~\eqref{main_eq} 
(see Appendix, Lemma~\ref{lem3.1}, for an analogous well-posedness statement under assumptions \textnormal{(F1)--(F4)}).

Note that the right-hand side of equation~\eqref{main_eq} also satisfies 
a linear growth condition.

\begin{fact}\label{fact:linear_growth}
Let 
$f : [0,(n+1)\tau]\times\mathbb{R}^d\times\mathbb{R}^d \to \mathbb{R}^d$
satisfy assumption~\textnormal{(E2)}. 
Then there exists a constant $M \ge 0$ such that for all 
$(t,y,z)\in[0,(n+1)\tau]\times\mathbb{R}^d\times\mathbb{R}^d$,
\begin{equation}\label{eq:linear_growth}
\|f(t,y,z)\| \le M(1+\|y\|)(1+\|z\|).
\end{equation}
\end{fact}

\begin{proof}
From the global Lipschitz condition~\textnormal{(E2)} we have
\[
\|f(t,y,z)-f(0,0,0)\|
   \le L\big(|t|+\|y\|+\|z\|\big)
   \le L\max\{1,(n+1)\tau\}\big(1+\|y\|+\|z\|\big).
\]
Consequently,
\begin{align}
\|f(t,y,z)\|
   &\le \|f(0,0,0)\| + L\max\{1,(n+1)\tau\}\big(1+\|y\|+\|z\|\big) \notag\\
   &\le \big(\|f(0,0,0)\|+L\max\{1,(n+1)\tau\}\big)(1+\|y\|+\|z\|) \notag\\
   &\le M(1+\|y\|)(1+\|z\|),
\end{align}
where we may take
\[
M := \|f(0,0,0)\| + L\max\{1,(n+1)\tau\}.
\]
\end{proof}



\subsection{Error of the Euler scheme}
\begin{lemma}\label{lem:boundedness}
Let $\tau \in (0,+\infty)$, $n \in \mathbb{N}$, $\eta \in \mathbb{R}^d$, and let 
$f$ satisfy assumptions \textnormal{(E1)}–\textnormal{(E2)}. 
There exist constants 
$\tilde{C}_0, \tilde{C}_1, \dots, \tilde{C}_n \in (0,+\infty)$, 
such that for all $N \in \mathbb{N}$ the following holds
\begin{equation}\label{eq:boundedness}
\max_{0 \le j \le n} \, \max_{0 \le k \le N} 
\| \tilde{y}_k^j \| \le \max_{0 \le j \le n} \tilde{C}_j.
\end{equation}
\end{lemma}

\begin{proof}
For $j=0$, the recursion reads
\[
\tilde{y}_{k+1}^0
  = \tilde{y}_k^0
  + h\,\tilde{f}\!\left(t_k^0,\tilde{y}_k^0,\tilde{\eta}\right),
\quad 
\tilde{y}_0^0=\tilde{\eta}.
\]
Since $\|\tilde{\eta}-\eta\|\le\delta\le1$, it follows that 
$\|\tilde{\eta}\|\le1+\|\eta\|$.  
Using the growth estimate 
$\|\tilde{f}(t,y,z)\|\le (M+\delta)(1+\|y\|)(1+\|z\|)$ we obtain
\[
\|\tilde{y}_{k+1}^0\|
  \le \|\tilde{y}_k^0\| 
  + h(M+1)(2+\|\eta\|)(1+\|\tilde{y}_k^0\|)
  = (1+hC_0)\|\tilde{y}_k^0\| + hC_0,
\]
where $C_0:=(M+1)(2+\|\eta\|)$.  
By the discrete Grönwall inequality,
\[
\|\tilde{y}_k^0\|
  \le (1+hC_0)^k\|\tilde{\eta}\|
     + \big((1+hC_0)^k-1\big)
  \le e^{\tau C_0}\|\tilde{\eta}\| + (e^{\tau C_0}-1)
  =: \tilde{C}_0.
\]

Assume now that for some $j\ge0$ we have
$\max_{0\le k\le N}\|\tilde{y}_k^j\|\le\tilde{C}_j$.  
Then, for $j+1$ we write
\[
\|\tilde{y}_{k+1}^{j+1}\|
  \le \|\tilde{y}_k^{j+1}\|
     + h(M+1)(1+\|\tilde{y}_k^{j+1}\|)(1+\|\tilde{y}_k^j\|) 
  \le (1+hC_j)\|\tilde{y}_k^{j+1}\| + hC_j,
\]
where $C_j := (M + 1)(1 + \tilde{C}_j)$.  
Applying the discrete Grönwall inequality again gives
\[
\|\tilde{y}_k^{j+1}\|
  \le (1+hC_j)^k\|\tilde{y}_0^{j+1}\|
     + ((1+hC_j)^k-1)
  \le e^{\tau C_j}\tilde{C}_j + (e^{\tau C_j}-1)
  =: \tilde{C}_{j+1}.
\]
Hence $\max_{0\le k\le N}\|\tilde{y}_k^{j+1}\|\le\tilde{C}_{j+1}$ for all $j$, 
which completes the proof.
\end{proof}

\begin{lemma}\label{thm:stability}
Let $\tau \in (0,+\infty)$, $n \in \mathbb{N}$, $\eta \in \mathbb{R}^d$, $\delta \in [0,1]$ and let 
$f$ satisfy assumptions \textnormal{(E1)}–\textnormal{(E2)}. Set $e_k^j := y_k^j - \tilde{y}_k^j$.
Then there exist constants $\bar K_0,\dots,\bar K_n>0$, such that for all $N\in\mathbb{N}$ the following holds
\begin{equation}\label{eq:error_bound}
\max_{0\le j\le n}\, \max_{0\le k\le N}\|e_k^j\|
   \le \max_{0\le j\le n}\bar K_j\,\delta.
\end{equation}
\end{lemma}

\begin{proof}
Define the errors $e_k^j=y_k^j-\tilde y_k^j$.  
From the recursive relations for $y_k^j$ and $\tilde y_k^j$, we obtain
\[
e_{k+1}^j
   = e_k^j + h\!\left[
   f(t_k^j,y_k^j,y_k^{j-1})
   - f(t_k^j,\tilde y_k^j,\tilde y_k^{j-1})
   -\tilde\delta_f(t_k^j,\tilde y_k^j,\tilde y_k^{j-1})
   \right].
\]
Using the global Lipschitz property of $f$ and the noise bound, one gets
\[
\|e_{k+1}^j\|
   \le (1+hL)\|e_k^j\|
       + hL\|e_k^{j-1}\|
       + h\delta(1+\|\tilde y_k^j\|)(1+\|\tilde y_k^{j-1}\|).
\]

From Lemma~\ref{lem:boundedness} we have uniform bounds 
$\|\tilde y_k^j\|\le\tilde C_j$, independent of $N$.  

For $j=0$, using $e_k^{-1}=\eta-\tilde\eta$ and $\|e_k^{-1}\|\le\delta$, we get
\[
\|e_{k+1}^0\|
\le (1+hL)\|e_k^0\| + hL\delta
   + h\delta(1+\|\tilde y_k^0\|)(1+\|\tilde\eta\|).
\]
By Lemma~\ref{lem:boundedness}, $\max_{0\le k\le N}\|\tilde y_k^0\|\le \tilde C_0$, and
$\|\tilde\eta\|\le 1+\|\eta\|$, hence
\[
\|e_{k+1}^0\|
\le (1+hL)\|e_k^0\| + h\delta\,\hat K_0,
\qquad
\hat K_0 := L + (1+\tilde C_0)(2+\|\eta\|).
\]
Applying the discrete Gr\"onwall inequality yields
\[
\max_{0\le k\le N}\|e_k^0\|
\le \bar K_0\delta
\]
for some $\bar K_0>0$ independent of $N$.

Assume inductively that for some $j\ge 0$,
\[
\max_{0\le k\le N}\|e_k^j\|\le \bar K_j\delta.
\]
Then, for layer $j+1$,
\[
\|e_{k+1}^{j+1}\|
\le (1+hL)\|e_k^{j+1}\| + hL\|e_k^j\|
   + h\delta(1+\|\tilde y_k^{j+1}\|)(1+\|\tilde y_k^j\|).
\]
Using Lemma~\ref{lem:boundedness} and the induction hypothesis,
\[
\|e_{k+1}^{j+1}\|
\le (1+hL)\|e_k^{j+1}\| + h\delta\,\hat K_j,
\]
where
\[
\hat K_j := L\bar K_j + (1+\tilde C_{j+1})(1+\tilde C_j).
\]
Since $e_0^{j+1}=e_N^j$, we have $\|e_0^{j+1}\|\le \bar K_j\delta$.
Applying the discrete Grönwall inequality (and in the case $L=0$ interpreting $\frac{e^{\tau L}-1}{L}$ as its limit equal to $\tau$) yields

\[
\max_{0\le k\le N}\|e_k^{j+1}\|
\le \left(e^{\tau L}\bar K_j + \frac{e^{\tau L}-1}{L}\hat K_j\right)\delta
=: \bar K_{j+1}\delta.
\]

Hence, by induction on $j$, inequality~\eqref{eq:error_bound} follows.
\end{proof}

\begin{remark}\label{rem:stability_constants}
It follows from the proofs of Lemma~\ref{thm:stability} and Lemma~\ref{lem:boundedness} 
that all constants $\tilde C_j$, $\bar K_j$, and $\hat K_j$ depend only on the model parameters 
$\tau$, $n$, $d$, $L$, and $\eta$. 
Consequently, the stability estimate~\eqref{eq:error_bound} holds uniformly with respect to the time-step size $h=\tau/N$. 
In particular, the Euler scheme remains uniformly stable with respect to perturbations as $N\to\infty$. 
Combined with the deterministic discretization error estimates under assumptions~\textnormal{(E1)}--\textnormal{(E2)}, 
this yields convergence of the perturbed Euler approximations to the exact solution as $h\to 0$ and $\delta\to 0$.
\end{remark}

\begin{theorem}\label{thm:lipschitz_global_error}
Let $\tau \in (0,+\infty)$, $n \in \mathbb{N}$, $\eta \in \mathbb{R}^d$, $\delta \in [0,1]$, and let 
$f$ satisfy assumptions \textnormal{(E1)}--\textnormal{(E2)}.
Then there exists a constant $C>0$, independent of $N$, such that
\begin{equation}\label{eq:global_error_noisy}
\sup_{t \in [0,(n+1)\tau]}
\| z(t) - l_{N, \delta}(t) \|
\le C (h+\delta).
\end{equation}
Here $z$ denotes the exact solution of~\eqref{main_eq}, and $l_{N, \delta}$ is the global
piecewise-linear interpolant obtained by combining local interpolants of the noisy Euler iterates
$\{\tilde y_k^j\}_{k=0}^N$ on subintervals $[j\tau,(j+1)\tau]$, $j=0,\dots,n$.
\end{theorem}

\begin{proof}
For each $j=0,1,\dots,n$ we introduce the noiseless Euler approximation
$\{y_k^j\}_{k=0}^N$ corresponding to the exact right-hand side $f$ and the
exact initial value $\eta$, and denote by $l_{N,0}^{\,j}$ the piecewise linear
interpolation of $\{y_k^j\}_{k=0}^N$ on $[j\tau,(j+1)\tau]$. Then
\[
z(t) - l_{N, \delta}^{\,j}(t)
= \bigl(z(t) - l_{N,0}^{\,j}(t)\bigr)
 + \bigl(l_{N,0}^{\,j}(t) - l_{N, \delta}^{\,j}(t)\bigr),
\qquad t\in[j\tau,(j+1)\tau].
\]

Under assumptions~\textnormal{(E1)}–\textnormal{(E2)}, the standard error
analysis for the Euler method applied to delay differential equations yields a
first-order convergence estimate on each subinterval $[j\tau,(j+1)\tau]$.
More precisely, by Theorem~2.5 in~\cite{nj_phd} (applied to the present
setting) there exists a constant $C_j^{(1)}\ge0$, independent of $N$, such that
\begin{equation}\label{eq:deterministic_part_paper}
\sup_{t\in[j\tau,(j+1)\tau]}
\bigl\| z(t) - l_{N,0}^{\,j}(t) \bigr\|
\le C_j^{(1)}\, h .
\end{equation}

Next, we estimate the difference between the noiseless and noisy Euler
approximations. Let $e_k^j := y_k^j - \tilde y_k^j$ for
$k=0,1,\dots,N$ and $j=0,1,\dots,n$. By Lemma~\ref{thm:stability}, there exist
constants $\bar K_0,\dots,\bar K_n\ge0$, such that
\begin{equation}\label{eq:stability_grid_paper}
\max_{0\le k\le N} \|e_k^j\|
\le \bar K_j\,\delta,
\qquad j=0,1,\dots,n.
\end{equation}
For any $t\in[t_k^j,t_{k+1}^j]$ we can write
\[
l_{N,0}^{\,j}(t)
= (1-\theta)\, y_k^j + \theta\, y_{k+1}^j,
\qquad
l_{N, \delta}^{\,j}(t)
= (1-\theta)\, \tilde y_k^j + \theta\, \tilde y_{k+1}^j
\]
for some $\theta = \theta(t)\in[0,1]$, and hence
\[
l_{N,0}^{\,j}(t) - l_{N, \delta}^{\,j}(t)
= (1-\theta)\, e_k^j + \theta\, e_{k+1}^j.
\]
Therefore,
\[
\bigl\| l_{N,0}^{\,j}(t) - l_{N, \delta}^{\,j}(t) \bigr\|
\le (1-\theta)\|e_k^j\| + \theta\|e_{k+1}^j\|
\le \max_{0\le m\le N}\|e_m^j\|,
\]
and by \eqref{eq:stability_grid_paper} we obtain
\begin{equation}\label{eq:stability_interp_paper}
\sup_{t\in[j\tau,(j+1)\tau]}
\bigl\| l_{N,0}^{\,j}(t) - l_{N, \delta}^{\,j}(t) \bigr\|
\le \bar K_j\,\delta.
\end{equation}

Combining \eqref{eq:deterministic_part_paper} and
\eqref{eq:stability_interp_paper} with the triangle inequality yields
\[
\sup_{t\in[j\tau,(j+1)\tau]}
\bigl\| z(t) - l_{N, \delta}^{\,j}(t) \bigr\|
\le C_j^{(1)} h + \bar K_j \delta
\le C_j (h+\delta),
\]
where we may take $C_j := C_j^{(1)} + \bar K_j$, $j=0,1,\dots,n$.
Taking the maximum over $j=0,\dots,n$ gives~\eqref{eq:global_error_noisy}.
\end{proof}

\section{Analysis under a Local One-Sided Lipschitz and Local Hölder Condition}
\label{sec:onesidedlipschitz}

\subsection{Problem formulation}
Let the right-hand side  
\(f \colon [0,+\infty)\times\mathbb{R}^{d}\times\mathbb{R}^{d}
      \to\mathbb{R}^{d}\)
in equation~\eqref{main_eq} satisfy the following conditions:

\begin{enumerate}
\item   [(F1)]
      \[
      f \in C\!\bigl([0,+\infty)\times\mathbb{R}^{d}\times
      \mathbb{R}^{d};\,\mathbb{R}^{d}\bigr).
      \]

\item [(F2)]
      There exists a constant \(K>0\) such that for every  
      \((t,y,z)\in[0,+\infty)\times\mathbb{R}^{d}\times\mathbb{R}^{d}\)
      \[
      \|f(t,y,z)\|\;\le\;
      K\,(1+\|y\|)\,(1+\|z\|).
      \]

\item [(F3)]
      There exists a constant \(H\in\mathbb{R}\) such that for all  
      \((t,z)\in[0,+\infty)\times\mathbb{R}^{d}\) and
      \(y_{1},y_{2}\in\mathbb{R}^{d}\)
      \[
      \bigl\langle y_{1}-y_{2},\,
      f(t,y_{1},z)-f(t,y_{2},z)\bigr\rangle
      \;\le\;
      H\,(1+\|z\|)\,\|y_{1}-y_{2}\|^{2}.
      \]

\item [(F4)]
      There exist constants \(L>0\) and
      exponents \(\alpha,\beta_{1},\beta_{2},\gamma\in(0,1]\) such that
      for all \(t_{1},t_{2}\in[0,+\infty)\) and
      \(y_{1},y_{2},z_{1},z_{2}\in\mathbb{R}^{d}\)
      \[
      \begin{aligned}
      \|f(t_{1},y_{1},z_{1})-f(t_{2},y_{2},z_{2})\|
      \;\le\;
      L\Bigl(&
      (1+\|y_{1}\|+\|y_{2}\|)\,
      (1+\|z_{1}\|+\|z_{2}\|)\,
      |t_{1}-t_{2}|^{\alpha}\\
      &\!+\,
      (1+\|z_{1}\|+\|z_{2}\|)\,
      \|y_{1}-y_{2}\|^{\beta_{1}}\\
      &\!+\,
      (1+\|z_{1}\|+\|z_{2}\|)\,
      \|y_{1}-y_{2}\|^{\beta_{2}}\\
      &\!+\,
      (1+\|y_{1}\|+\|y_{2}\|)\,
      \|z_{1}-z_{2}\|^{\gamma}\Bigr).
      \end{aligned}
      \]
\end{enumerate}
Assumption~\textnormal{(F2)} provides a \emph{global linear-growth} bound for the right-hand side function $f$, while the usual global Lipschitz requirement is replaced by the \emph{one-sided} condition~\textnormal{(F3)}. Condition~\textnormal{(F4)} is commonly referred to as a \emph{local H\"older} condition. As noted in~\cite{nj_phd}, these assumptions (especially \textnormal{(F3)} and \textnormal{(F4)}) are motivated by a real-life model describing the evolution of dislocation density. Under assumptions \textnormal{(F1)-(F4)}, problem~\eqref{main_eq} admits a unique solution \(z=z(t)\) on the whole interval (see Lemma~\ref{lem3.1}).%
%
%
%
\begin{remark}
If \(f\) satisfies assumptions~\textnormal{(E1)}--\textnormal{(E2)}, then it also satisfies~\textnormal{(F1)}--\textnormal{(F4)} with \(\alpha=\beta_{1}=\beta_{2}=\gamma=1\). We treat these two settings separately because the stronger assumptions~\textnormal{(E1)}--\textnormal{(E2)} lead to sharper convergence rates for the Euler scheme.
\end{remark}

\subsection{Error of the Euler scheme}

%
\begin{lemma}\label{lem:boundedness_onesided}
Let $\tau \in (0,+\infty)$, $n \in \mathbb{N}$, $\eta \in \mathbb{R}^d$, $\delta \in [0,1]$ and let 
$f$ satisfy assumptions \textnormal{(F1)}–\textnormal{(F2)} 
Then there exist constants 
$\tilde{C}_0, \dots, \tilde{C}_n > 0$ such that for all $N \in \mathbb{N}$
\[
\max_{0 \le j \le n} \max_{0 \le k \le N} 
\|\tilde{y}_k^j\| \le \max_{0 \le j \le n} \tilde{C}_j.
\]
\end{lemma}

\begin{proof}
For $j=0$, the recursion reads
\[
\tilde{y}_{k+1}^0
  = \tilde{y}_k^0
  + h\,\tilde{f}(t_k^0,\tilde{y}_k^0,\tilde{\eta}),
\qquad 
\tilde{y}_0^0 = \tilde{\eta}, \quad \|\tilde{\eta}-\eta\| \le \delta \le 1.
\]
From (F2) and Assumption~\ref{ass:noise} we have
\[
\|\tilde{f}(t,y,z)\|
 = \|f(t,y,z) + \tilde{\delta}_f(t,y,z)\|
 \le (K + \delta)(1 + \|y\|)(1 + \|z\|).
\]
Hence,
\[
\|\tilde{y}_{k+1}^0\|
 \le \|\tilde{y}_k^0\|
   + h (K + \delta)(1 + \|\tilde{y}_k^0\|)(1 + \|\tilde{\eta}\|)
 \le (1 + h C_0)\|\tilde{y}_k^0\| + h C_0,
\]
where $C_0 := (K + \delta)(2 + \|\eta\|)$.  
By the discrete Grönwall inequality,
\[
\|\tilde{y}_k^0\|
 \le e^{\tau C_0}\|\tilde{\eta}\| + (e^{\tau C_0} - 1)
 =: \tilde{C}_0.
\]

Assume now that for some $j \ge 0$ we have
$\max_{0 \le k \le N} \|\tilde{y}_k^{j}\| \le \tilde{C}_j$.
Then
\[
\|\tilde{y}_{k+1}^{\,j+1}\|
 \le \|\tilde{y}_k^{\,j+1}\|
   + h (K + \delta)(1 + \|\tilde{y}_k^{\,j+1}\|)(1 + \|\tilde{y}_k^{\,j}\|)
 \le (1 + h C_j)\|\tilde{y}_k^{\,j+1}\| + h C_j,
\]
where $C_j := (K + \delta)(1 + \tilde{C}_j)$.  The discrete Gr\"onwall inequality yields again
\[
\|\tilde{y}_k^{\,j+1}\|
 \le e^{\tau C_j}\tilde{C}_j + (e^{\tau C_j} - 1)
 =: \tilde{C}_{j+1}.
\]
Induction in $j$ completes the proof and gives
\[
\max_{0 \le j \le n} \max_{0 \le k \le N} \|\tilde{y}_k^j\|
 \le \max_{0 \le j \le n} \tilde{C}_j,
\]
\end{proof}

\vspace{0.5cm}

\begin{lemma}\label{thm:euler_error}
Let $\tau \in (0,+\infty)$, $n \in \mathbb{N}$, $\eta \in \mathbb{R}^d$ and $\delta \in (0,1]$. Let 
$f$ satisfy assumptions \textnormal{(F1)}–\textnormal{(F4)} and set $e_k^j := y_k^j - \tilde{y}_k^j$.

\begin{enumerate}[label=(\Alph*)]
\item If $\gamma = 1$, there exist constants $C_0, C_1, \dots, C_n > 0$ such that for all $N \ge 2\lceil \tau \rceil$ and each $j = 0,1,\dots,n$,
\[
\max_{0 \le k \le N} \| e_k^{\,j} \|
\le
C_j\bigl(h^{1/2}+\delta\bigr).
\]

\item If $\gamma \in (0,1)$, then there exist constants $C_0, C_1, \dots, C_n > 0$ such that for all $N \ge 2\lceil \tau \rceil$ 
\[
\max_{0 \le k \le N} \| e_k^{\,0} \|
\;\le\;
C_0\bigl(h^{1/2} + \delta^{\gamma}\bigr),
\]
and for each $j = 0, 1, \dots, n-1$,
\[
\max_{0 \le k \le N} \| e_k^{\,j+1} \|
\;\le\;
C_{j+1}\!\left(
\sum_{\ell = 0}^{j+1} h^{\frac{\gamma^{\ell}}{2}}
\;+\;
\sum_{\ell = 0}^{j+2} \delta^{\gamma^{\ell}}
\right)\!.
\]
\end{enumerate}
\end{lemma}

\begin{proof}
Fix $N\in\mathbb{N}$ and set $h=\tau/N$. Assume $N\ge 2\lceil\tau\rceil$ so that $h\le 1/2$.
For $k=0,1,\dots,N-1$ and $j=0,1,\dots,n$ define
\[
U_k^j=(t_k^j,\,y_k^j,\,y_k^{j-1}), \qquad
\tilde U_k^j=(t_k^j,\,\tilde y_k^j,\,\tilde y_k^{j-1}).
\]
Subtracting the Euler recursions and using $\tilde f=f+\tilde\delta_f$ we obtain
\[
e_{k+1}^j
= e_k^j + h\Big(f(U_k^j)-f(\tilde U_k^j)\Big)
- h\,\tilde\delta_f(\tilde U_k^j).
\]
Add and subtract $f(t_k^j,\tilde y_k^j,y_k^{j-1})$ and introduce
\[
R_k^j := f(t_k^j,y_k^j,y_k^{j-1})-f(t_k^j,\tilde y_k^j,y_k^{j-1}),
\]
\[
L_k^j := f(t_k^j,\tilde y_k^j,y_k^{j-1})-f(t_k^j,\tilde y_k^j,\tilde y_k^{j-1})
\;-\;\tilde\delta_f(\tilde U_k^j).
\]
Then
\begin{equation}\label{eq:basic-identity}
e_{k+1}^j - h\,L_k^j = e_k^j + h\,R_k^j.
\end{equation}

and 
\[
\|e_{k+1}^j-hL_k^j\|^2=\|e_k^j+hR_k^j\|^2.
\]
Hence
\[
\|e_{k+1}^j\|^2 -2h\langle e_{k+1}^j,L_k^j\rangle
\le
\|e_k^j\|^2 +2h\langle e_k^j,R_k^j\rangle + h^2\|R_k^j\|^2.
\]
Using Cauchy--Schwarz and Young's inequalities we have
\[
2h\langle e_{k+1}^j, L_k^j\rangle
\le 2h|\langle e_{k+1}^j, L_k^j\rangle|
\le 2h\|e_{k+1}^j\|\,\|L_k^j\|
\le h\|e_{k+1}^j\|^2 + h\|L_k^j\|^2.
\]
Hence,
\[
(1-h)\|e_{k+1}^j\|^2
\le
\|e_k^j\|^2 +2h\langle e_k^j,R_k^j\rangle + h^2\|R_k^j\|^2 + h\|L_k^j\|^2.
\]

Since $h\le 1/2$, we have
\[
(1-h)^{-1}\le 1+2h.
\]
Therefore,
\begin{equation}\label{eq:energy-pre}
\|e_{k+1}^j\|^2
\le
(1+2h)\Bigl(
\|e_k^j\|^2 +2h\langle e_k^j,R_k^j\rangle + h^2\|R_k^j\|^2 + h\|L_k^j\|^2
\Bigr).
\end{equation}

We use the boundedness of both Euler trajectories on each layer.
More precisely, Lemma~\ref{lem:boundedness_onesided} yields
$\max_{j,k}\|\tilde y_k^j\|\le \max_{0\le j\le n}\tilde C_j$.
The same argument with $\delta=0$ gives analogous bounds for $\{y_k^j\}$.
Therefore, all factors of the form $(1+\|y_k^j\|+\|\tilde y_k^j\|)$ can be absorbed
into constants depending only on $\tau,n,d$ and the parameters in \textnormal{(F1)}--\textnormal{(F4)}.

\smallskip
\noindent\emph{(i) The one-sided Lipschitz term.}
By \textnormal{(F3)}, with $z=y_k^{j-1}$ fixed,
\[
\langle e_k^j,R_k^j\rangle
=
\big\langle y_k^j-\tilde y_k^j,\,
f(t_k^j,y_k^j,y_k^{j-1})-f(t_k^j,\tilde y_k^j,y_k^{j-1})\big\rangle
\le C\,\|e_k^j\|^2,
\]
hence $2h\langle e_k^j,R_k^j\rangle \le C h\|e_k^j\|^2$.

\smallskip
\noindent\emph{(ii) The $R_k^j$-term.}
Using \textnormal{(F4)} with $t_1=t_2=t_k^j$ and $z_1=z_2=y_k^{j-1}$, we obtain
\[
\|R_k^j\|\le C\big(\|e_k^j\|^{\beta_1}+\|e_k^j\|^{\beta_2}\big)
\le C(1+\|e_k^j\|),
\]
where we used $x^\rho\le 1+x$ for $x\ge 0$ and $\rho\in(0,1]$.
Thus,
\[
h^2\|R_k^j\|^2 \le C h^2 + C h^2\|e_k^j\|^2.
\]

\smallskip
\noindent\emph{(iii) The $L_k^j$-term (delay Hölder + noise).}
By the noise assumption and boundedness of $\tilde y$,
\[
\|\tilde\delta_f(\tilde U_k^j)\|\le \delta(1+\|\tilde y_k^j\|)(1+\|\tilde y_k^{j-1}\|)\le C\delta.
\]
Moreover, by \textnormal{(F4)} in the $z$-variable (exponent $\gamma$),
\[
\|f(t_k^j,\tilde y_k^j,y_k^{j-1})-f(t_k^j,\tilde y_k^j,\tilde y_k^{j-1})\|
\le C\|y_k^{j-1}-\tilde y_k^{j-1}\|^\gamma
= C\|e_k^{j-1}\|^\gamma.
\]
Hence,
\[
\|L_k^j\|\le C\big(\delta+\|e_k^{j-1}\|^\gamma\big),
\qquad
\|L_k^j\|^2\le C\big(\delta^2+\|e_k^{j-1}\|^{2\gamma}\big).
\]

Inserting the above bounds into \eqref{eq:energy-pre} and expanding the factor $(1+2h)$, we obtain terms of the form
\[
(1+2h)\bigl(1+Ch+Ch^2\bigr)\|e_k^j\|^2 + C(1+2h)h^2 + C(1+2h)h\bigl(\delta^2+\|e_k^{j-1}\|^{2\gamma}\bigr).
\]
Since $h\le 1/2$, we have $h^2\le h$, and therefore
\[
(1+2h)\bigl(1+Ch+Ch^2\bigr)\le 1+a_j h
\]
for a suitable constant $a_j>0$ independent of $N$. Hence, for each fixed $j$,
\begin{equation}\label{eq:star}
\|e_{k+1}^j\|^2
\le (1+a_j h)\|e_k^j\|^2
+ b_j h^2
+ c_j h\big(\delta^2+\|e_k^{j-1}\|^{2\gamma}\big),
\qquad k=0,\dots,N-1,
\end{equation}
with constants $a_j,b_j,c_j\ge 0$ independent of $N$.

Let $E_j:=\max_{0\le k\le N}\|e_k^j\|$. Since $e_0^j=e_N^{j-1}$, we have $\|e_0^j\|\le E_{j-1}$. Applying the discrete Gr\"onwall inequality
(Lemma~\ref{gronwall}) to \eqref{eq:star} gives
\begin{equation}\label{eq:Ej2}
E_j^2 \le C_j\Big(E_{j-1}^2 + h + \delta^2 + E_{j-1}^{2\gamma}\Big),
\qquad j=0,1,\dots,n,
\end{equation}
with constants $C_j$ independent of $N$.
Finally, on the history layer $j=-1$ we have $e_k^{-1}=\eta-\tilde\eta$, hence
$E_{-1}\le \delta$.

\smallskip
\noindent\emph{Case $\gamma=1$.}
Then \eqref{eq:Ej2} becomes $E_j^2 \le C_j(E_{j-1}^2+h+\delta^2)$.
Starting from $E_{-1}\le \delta$ we obtain inductively
$E_j^2\le \tilde C_j(h+\delta^2)$, hence
\[
E_j \le C_j\big(h^{1/2}+\delta\big),
\qquad j=0,1,\dots,n.
\]

\smallskip
\noindent\emph{Case $\gamma\in(0,1)$.}
From \eqref{eq:Ej2} we get
$E_0^2\le C(h+\delta^{2\gamma})$, hence $E_0\le C(h^{1/2}+\delta^\gamma)$.
Next, taking square roots in \eqref{eq:Ej2} yields the convenient form
\[
E_j \le C_j\Big(E_{j-1}+h^{1/2}+\delta+E_{j-1}^{\gamma}\Big).
\]
Using subadditivity $(x+y)^\gamma\le x^\gamma+y^\gamma$ for $\gamma\in(0,1)$ (and iterating it over finite sums), one checks by induction that the exponents propagate,
which leads to
\[
E_{j+1}\le
C_{j+1}\!\left(
\sum_{\ell = 0}^{j+1} h^{\frac{\gamma^{\ell}}{2}}
\;+\;
\sum_{\ell = 0}^{j+2} \delta^{\gamma^{\ell}}
\right),
\qquad j=0,1,\dots,n-1.
\]
This completes the proof.
\end{proof}

\begin{remark}
For $\delta=0$, we have $\tilde f=f$ and $\tilde\eta=\eta$, hence the noisy and noiseless Euler schemes coincide and $e_k^j\equiv 0$ for all $j,k$.
\end{remark}

\begin{theorem}
\label{thm:noisy_holder}
Let $\tau \in (0,+\infty)$, $n \in \mathbb{N}$, $\eta \in \mathbb{R}^d$ and $\delta \in [0,1]$. Let 
$f$ satisfy assumptions \textnormal{(F1)}–\textnormal{(F4)}. Then there exist constants
$\bar C_0,\bar C_1,\dots,\bar C_n \ge 0$
such that for all $N \ge 2\lceil \tau \rceil $ 
the following holds:

\begin{enumerate}[label=(\Alph*)]
\item \emph{If $\gamma = 1$, then}
\begin{equation}\label{eq:noisy_holder_gamma1_j0}
\sup_{t\in[0,\tau]}
\bigl\|\phi_0(t) - l_{N,\delta}^0(t)\bigr\|
\;\le\;
\bar C_0\bigl(h^{\alpha} + h^{\beta_1} + h^{\beta_2} + h^{1/2} + \delta\bigr),
\end{equation}
and for each $j=1,2,\dots,n$,
\begin{equation}\label{eq:noisy_holder_gamma1_j}
\sup_{t\in[j\tau,(j+1)\tau]}
\bigl\|\phi_j(t) - l_{N,\delta}^j(t)\bigr\|
\;\le\;
\bar C_j\bigl(h^{1/2} + h^{\alpha} + h^{\beta_1} + h^{\beta_2} + \delta\bigr).
\end{equation}

\item \emph{If $\gamma \in (0,1)$, then}
\begin{equation}\label{eq:noisy_holder_gamma<1_j0}
\sup_{t\in[0,\tau]}
\bigl\|\phi_0(t) - l_{N,\delta}^0(t)\bigr\|
\;\le\;
\bar C_0\bigl(h^{\alpha\wedge\gamma} + h^{\beta_1} + h^{\beta_2} + h^{1/2} 
             + \delta^{\gamma}\bigr),
\end{equation}
and for each $j=1,2,\dots,n$,
\begin{equation}\label{eq:noisy_holder_gamma<1_j}
\sup_{t\in[j\tau,(j+1)\tau]}
\bigl\|\phi_j(t) - l_{N,\delta}^j(t)\bigr\|
\;\le\;
\bar C_j\left(
\sum_{\ell=1}^{j}
 \bigl(
   h^{\frac{\gamma^{\ell}}{2}}
   + h^{\gamma^{\ell}(\alpha\wedge\gamma)}
   + h^{\beta_1\gamma^{\ell}}
   + h^{\beta_2\gamma^{\ell}}
 \bigr)
\;+\;
\sum_{\ell=0}^{j+2} \delta^{\gamma^{\ell}}
\right).
\end{equation}
\end{enumerate}
where $\phi_j$ denotes the solution $z=z(t)$ for $t\in [j\tau,(j+1)\tau]$
and $l_{N,\delta}^j$ denotes the piecewise-linear interpolation of the noisy Euler iterates.
\end{theorem}

\begin{remark}
In particular, for $\gamma=1$ the noise does not accumulate over the 
subintervals and the global error of the noisy Euler scheme is of order
$O\bigl(h^{1/2} + h^{\alpha} + h^{\beta_1} + h^{\beta_2} + \delta\bigr)$
on each interval $[j\tau,(j+1)\tau]$.
\end{remark}


\begin{proof}
Fix $n\in\mathbb{N}_0$ and $N\in\mathbb{N}$. For each $j=0,1,\dots,n$ let
$\{y_k^j\}_{k=0}^N$ denote the Euler iterates corresponding to the exact
right-hand side $f$ (without noise) and the exact initial value $\eta$, and let
$l_N^j$ be their piecewise linear interpolation on $[j\tau,(j+1)\tau]$. The
noisy iterates are denoted by $\{\tilde y_k^j\}_{k=0}^N$, and
$l_{N,\delta}^j$ is the corresponding interpolation. Moreover, we write
$\phi_j$ for the restriction of the exact solution $z$ of
\eqref{main_eq} to $[j\tau,(j+1)\tau]$.

For $t\in[j\tau,(j+1)\tau]$ we decompose
\[
\phi_j(t) - l_{N,\delta}^j(t)
= \bigl(\phi_j(t) - l_N^j(t)\bigr)
 + \bigl(l_N^j(t) - l_{N,\delta}^j(t)\bigr).
\]

\medskip
\noindent
\emph{Deterministic Euler error (noiseless case).}
Under assumptions~\textnormal{(F1)}–\textnormal{(F4)}, the error analysis of the Euler scheme for DDEs with one-sided Lipschitz condition and local H\"older
regularity (see \cite[Theorems~3.2–3.3]{ncz_pp_pm}) yields the following bounds
for the scheme based on the exact right-hand side $f$:
\begin{enumerate}[label=(D\arabic*)]
\item If $\gamma=1$, then there exist constants
$C_0^{\mathrm{det}},\dots,C_n^{\mathrm{det}}\ge0$ such that
\begin{align}
\sup_{t\in[0,\tau]}
\bigl\|\phi_0(t) - l_N^0(t)\bigr\|
&\le C_0^{\mathrm{det}}\bigl(
  h^{\alpha} + h^{\beta_1} + h^{\beta_2} + h^{1/2}\bigr),
\label{eq:det-gamma1-j0}
\\
\sup_{t\in[j\tau,(j+1)\tau]}
\bigl\|\phi_j(t) - l_N^j(t)\bigr\|
&\le C_j^{\mathrm{det}}\bigl(
  h^{1/2} + h^{\alpha} + h^{\beta_1} + h^{\beta_2}\bigr),
\quad j=1,\dots,n.
\label{eq:det-gamma1-j}
\end{align}

\item If $\gamma\in(0,1)$, then there exist constants
$C_0^{\mathrm{det}},\dots,C_n^{\mathrm{det}}\ge0$ such that
\begin{align}
\sup_{t\in[0,\tau]}
\bigl\|\phi_0(t) - l_N^0(t)\bigr\|
&\le C_0^{\mathrm{det}}\bigl(
  h^{\alpha\wedge\gamma} + h^{\beta_1} + h^{\beta_2} + h^{1/2}\bigr),
\label{eq:det-gamma<1-j0}
\\
\sup_{t\in[j\tau,(j+1)\tau]}
\bigl\|\phi_j(t) - l_N^j(t)\bigr\|
&\le C_j^{\mathrm{det}}
  \sum_{\ell=1}^{j}
 \bigl(
   h^{\frac{\gamma^{\ell}}{2}}
   + h^{\gamma^{\ell}(\alpha\wedge\gamma)}
   + h^{\beta_1\gamma^{\ell}}
   + h^{\beta_2\gamma^{\ell}}
 \bigr),
\quad j=1,\dots,n.
\label{eq:det-gamma<1-j}
\end{align}
\end{enumerate}
These estimates correspond to the case $\delta=0$ (exact information model).

\medskip
\noindent
\emph{Error between noiseless and noisy Euler schemes.}
Set $e_k^j := y_k^j - \tilde y_k^j$ for $k=0,\dots,N$ and $j=0,\dots,n$. By
Lemma~\ref{lem:boundedness_onesided}, the noisy terms are uniformly bounded
in $j,k$, and using (F3), (F4) together with Assumption~\ref{ass:noise} one
obtains (cf. Lemma~\ref{thm:euler_error}) the following grid–error bounds:
\begin{enumerate}[label=(E\arabic*)]
\item If $\gamma=1$, then there exist constants $C_0^e,\dots,C_n^e>0$ such that
\begin{equation}\label{eq:euler-error-gamma1}
\max_{0\le k\le N}\|e_k^{\,j}\|
\;\le\;
C_j^e\bigl(h^{1/2} + \delta\bigr),
\qquad j=0,1,\dots,n.
\end{equation}

\item If $\gamma\in(0,1)$, then there exist constants $C_0^e,\dots,C_n^e>0$
such that
\begin{align}
\max_{0\le k\le N}\|e_k^{\,0}\|
&\le C_0^e\bigl(h^{1/2} + \delta^{\gamma}\bigr),
\label{eq:euler-error-gamma<1-0}
\\
\max_{0\le k\le N}\|e_k^{\,j}\|
&\le C_j^e\left(
\sum_{\ell = 0}^{j} h^{\frac{\gamma^{\ell}}{2}}
\;+\;
\sum_{\ell = 0}^{j+1} \delta^{\gamma^{\ell}}
\right),
\qquad j=1,\dots,n.
\label{eq:euler-error-gamma<1-j}
\end{align}
\end{enumerate}

For fixed $j$ and $t\in[t_k^j,t_{k+1}^j]$ we can write
\[
l_N^{j}(t)
= (1-\theta)\,y_k^j + \theta\,y_{k+1}^j,
\qquad
l_{N,\delta}^{j}(t)
= (1-\theta)\,\tilde y_k^j + \theta\,\tilde y_{k+1}^j,
\]
for some $\theta = \theta(t)\in[0,1]$. Hence
\[
l_N^{j}(t) - l_{N,\delta}^{j}(t)
= (1-\theta)\,e_k^j + \theta\,e_{k+1}^j,
\]
and therefore
\begin{equation}\label{eq:interp-diff}
\sup_{t\in[j\tau,(j+1)\tau]}
\bigl\|l_N^{j}(t) - l_{N,\delta}^{j}(t)\bigr\|
\le \max_{0\le k\le N}\|e_k^j\|.
\end{equation}

\medskip
\noindent
\emph{Combination of estimates.}
Using the triangle inequality we obtain
\[
\sup_{t\in[j\tau,(j+1)\tau]}
\bigl\|\phi_j(t) - l_{N,\delta}^j(t)\bigr\|
\le
\sup_{t\in[j\tau,(j+1)\tau]}
\bigl\|\phi_j(t) - l_N^j(t)\bigr\|
+
\sup_{t\in[j\tau,(j+1)\tau]}
\bigl\|l_N^{j}(t) - l_{N,\delta}^{j}(t)\bigr\|.
\]

\smallskip
\noindent
\textbf{Case $\gamma = 1$.}
For $j=0$, combining \eqref{eq:det-gamma1-j0},
\eqref{eq:euler-error-gamma1} and~\eqref{eq:interp-diff} yields
\[
\sup_{t\in[0,\tau]}
\bigl\|\phi_0(t) - l_{N,\delta}^0(t)\bigr\|
\le
C_0^{\mathrm{det}}\bigl(h^{\alpha} + h^{\beta_1} + h^{\beta_2} + h^{1/2}\bigr)
+ C_0^e\bigl(h^{1/2} + \delta\bigr)
\le
\bar C_0\bigl(h^{\alpha} + h^{\beta_1} + h^{\beta_2} + h^{1/2} + \delta\bigr).
\]
For $j\ge1$, using \eqref{eq:det-gamma1-j}, \eqref{eq:euler-error-gamma1}
and~\eqref{eq:interp-diff} we obtain
\[
\sup_{t\in[j\tau,(j+1)\tau]}
\bigl\|\phi_j(t) - l_{N,\delta}^j(t)\bigr\|
\le
C_j^{\mathrm{det}}\bigl(h^{1/2} + h^{\alpha} + h^{\beta_1} + h^{\beta_2}\bigr)
+ C_j^e\bigl(h^{1/2} + \delta\bigr)
\le
\bar C_j\bigl(h^{1/2} + h^{\alpha} + h^{\beta_1} + h^{\beta_2} + \delta\bigr),
\]
which proves \eqref{eq:noisy_holder_gamma1_j0}–\eqref{eq:noisy_holder_gamma1_j}.

\smallskip
\noindent
\textbf{Case $\gamma\in(0,1)$.}
For $j=0$, we combine \eqref{eq:det-gamma<1-j0},
\eqref{eq:euler-error-gamma<1-0} and~\eqref{eq:interp-diff} to obtain
\[
\sup_{t\in[0,\tau]}
\bigl\|\phi_0(t) - l_{N,\delta}^0(t)\bigr\|
\le
C_0^{\mathrm{det}}\bigl(h^{\alpha\wedge\gamma} + h^{\beta_1} + h^{\beta_2} + h^{1/2}\bigr)
+ C_0^e\bigl(h^{1/2} + \delta^{\gamma}\bigr)
\le
\bar C_0\bigl(h^{\alpha\wedge\gamma} + h^{\beta_1} + h^{\beta_2} + h^{1/2}
             + \delta^{\gamma}\bigr),
\]
which gives \eqref{eq:noisy_holder_gamma<1_j0}.

For $j\ge1$, by \eqref{eq:det-gamma<1-j},
\eqref{eq:euler-error-gamma<1-j} and~\eqref{eq:interp-diff},
\[
\sup_{t\in[j\tau,(j+1)\tau]}
\bigl\|\phi_j(t) - l_{N,\delta}^j(t)\bigr\|
\le
C_j^{\mathrm{det}}
  \sum_{\ell=1}^{j}
 \bigl(
   h^{\frac{\gamma^{\ell}}{2}}
   + h^{\gamma^{\ell}(\alpha\wedge\gamma)}
   + h^{\beta_1\gamma^{\ell}}
   + h^{\beta_2\gamma^{\ell}}
 \bigr)
+
C_j^e\left(
\sum_{\ell = 0}^{j} h^{\frac{\gamma^{\ell}}{2}}
\;+\;
\sum_{\ell = 0}^{j+1} \delta^{\gamma^{\ell}}
\right).
\]
Since the sums in the estimate for $\max_k\|e_k^j\|$ start from $\ell=0$
while the deterministic contributions start from $\ell=1$, we can absorb
all terms into a single constant $\bar C_j$ and slightly extend the indices
in the $\delta$–sum, obtaining
\[
\sup_{t\in[j\tau,(j+1)\tau]}
\bigl\|\phi_j(t) - l_{N,\delta}^j(t)\bigr\|
\le
\bar C_j\left(
\sum_{\ell=1}^{j}
 \bigl(
   h^{\frac{\gamma^{\ell}}{2}}
   + h^{\gamma^{\ell}(\alpha\wedge\gamma)}
   + h^{\beta_1\gamma^{\ell}}
   + h^{\beta_2\gamma^{\ell}}
 \bigr)
\;+\;
\sum_{\ell=0}^{j+2} \delta^{\gamma^{\ell}}
\right),
\]
which is precisely \eqref{eq:noisy_holder_gamma<1_j}. This completes the proof.
\end{proof}

\section{Numerical experiments}

We focus on two aspects: convergence with respect to the stepsize \(h\) and error propagation across successive delay intervals. Computations are carried out for multiple noise levels \(\delta\) and two values of the delay exponent \(\gamma \in \{0.225, 1.0\}\).

\subsection{Test equations}

We consider four scalar delay differential equations of the form
\[
z'(t) = f(t, z(t), z(t-\tau)), 
\qquad t \ge 0,
\]
with constant delay $\tau>0$ and constant history $z(t) \equiv \eta$ for $t \in [-\tau,0]$.
In all experiments we use
\[
\eta = z_0 = 0.05854.
\]

The right-hand sides $f_1,\dots,f_4$ are defined as follows (for $(t,y,z) \in [0,\infty)\times\mathbb{R}\times\mathbb{R}$).

\paragraph{Example~1 (metal-type drift with $\mathrm{sgn}(y)$).}
We set
\[
f_1(t,y,z)
=
A
- B\,\mathrm{sgn}(y)\,|y|
- C\,\mathrm{sgn}(y)\,|y|^{\rho}\,|z|^{\gamma}
+ D\,y\,|z|^{\gamma},
\]
with parameters
\[
A = 1.7137, \qquad
B = 0.7769, \qquad
C = 0.5895, \qquad
D = -0.82615, \qquad
\rho = 0.973,
\]
and $\gamma \in \{0.225,\,1.0\}$.

\paragraph{Example~2 (one-sided power nonlinearity in $y$ and Hölder term in $z$).}
We define
\[
f_2(t,y,z)
=
\sin(t)
+ c\,|z|^{\gamma}
- a\,y
- (1+|z|)\,\bigl(\max\{y,0\}\bigr)^{\beta},
\]
with parameters
\[
\beta = 0.7, \qquad
a = 0.2, \qquad
c = 2.0, \qquad
\gamma \in \{0.225,\,1.0\}.
\]

\paragraph{Example~3 (symmetric power in $y$ and Hölder term in $z$).}
Here
\[
f_3(t,y,z)
=
\sin(t)
+ c\,\mathrm{sgn}(z)\,|z|^{\gamma}
- a\,y
- (1+|z|)\,\mathrm{sgn}(y)\,|y|^{\beta},
\]
with the same parameters as in Example~2:
\[
\beta = 0.7, \qquad
a = 0.2, \qquad
c = 2.0, \qquad
\gamma \in \{0.225,\,1.0\}.
\]

\paragraph{Example~4 (modified oscillatory delay equation).}
Finally, we consider
\[
f_4(t,y,z)
=
3\,\mathrm{sgn}(z)\,|z|^{\gamma}\,\sin(\lambda t),
\]
with
\[
\lambda = 1.0, 
\qquad
\gamma \in \{0.225,\,1.0\}.
\]

In all four examples $f_i$ satisfies assumptions \textnormal{(F1)}–\textnormal{(F4)} and is not globally Lipschitz in the delayed argument when $\gamma \in (0,1)$.

\subsection{Perturbed information model}

In all noisy experiments we work with the perturbed right-hand side
\[
\tilde{f}_\delta(t,y,z)
=
f(t,y,z) + \tilde{\delta}_f(t,y,z),
\]
where the perturbation is given by
\[
\tilde{\delta}_f(t,y,z)
=
\delta\, U(-1,1)\, \bigl(1 + |y| + |z|\bigr),
\]
$\delta \in [0,1]$ is the noise level and $U(-1,1)$ denotes a random variable
uniformly distributed on the interval $[-1,1]$.
Independent samples of $U(-1,1)$ are used at each time step and for each trajectory.
We consider the noise levels
\[
\delta \in \{0.0,\,0.01,\,0.05,\,0.1,\,0.2,\,0.5,\,0.75,\,1.0\}.
\]
Although the perturbation is random in simulations, it satisfies the deterministic noise bound
from Assumption~\ref{ass:noise} pathwise, since $|U(-1,1)|\le 1$ and
\[
1+|y|+|z| \le (1+|y|)(1+|z|).
\]
Hence the theoretical bounds apply to each realization.

\subsection{Convergence experiment with respect to the step size}

We set \(\tau=20\), \(n=9\), and \(T=(n+1)\tau\). For
\[
N\in\{\lfloor 100\cdot 1.3^i\rfloor:\ i=0,\dots,12\},\qquad h=\tau/N,
\]
we compute Euler trajectories and a reference trajectory on a refined grid \(N_{\mathrm{ref}}=50N\) (with \(\delta=0\)), then sample the reference back to the coarse grid. For each \((f,\gamma,\delta,N)\), we generate \(50\) noisy trajectories and record the interval-wise and cumulative supremum errors relative to the reference.

The convergence plots in Fig.~\ref{fig:convergence-4x2} show the expected transition between a discretization-dominated regime and a noise-dominated regime. For small \(\delta\), errors decrease with \(h\); for larger \(\delta\), a plateau appears. The plateau is reached earlier and at a higher level when \(\gamma<1\), consistently with the theoretical bounds.

\begin{figure}[htbp]
    \centering

    \begin{subfigure}[t]{0.4\textwidth}
        \centering
        \includegraphics[width=\textwidth]{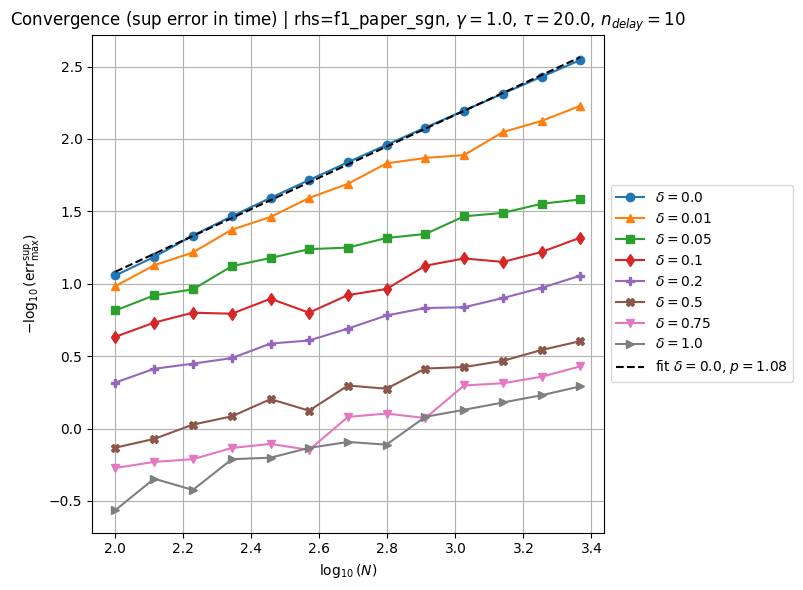}
        \caption{$f_1$, $\gamma = 1.0$}
        \label{fig:f1-gamma-1}
    \end{subfigure}\hfill
    \begin{subfigure}[t]{0.4\textwidth}
        \centering
        \includegraphics[width=\textwidth]{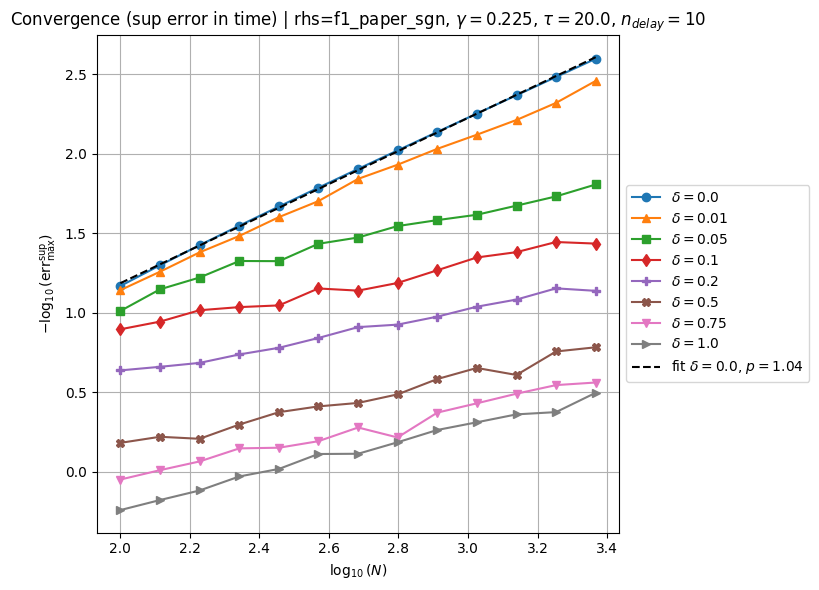}
        \caption{$f_1$, $\gamma = 0.225$}
        \label{fig:f1-gamma-0.225}
    \end{subfigure}

    \vspace{0.4cm}

    \begin{subfigure}[t]{0.4\textwidth}
        \centering
        \includegraphics[width=\textwidth]{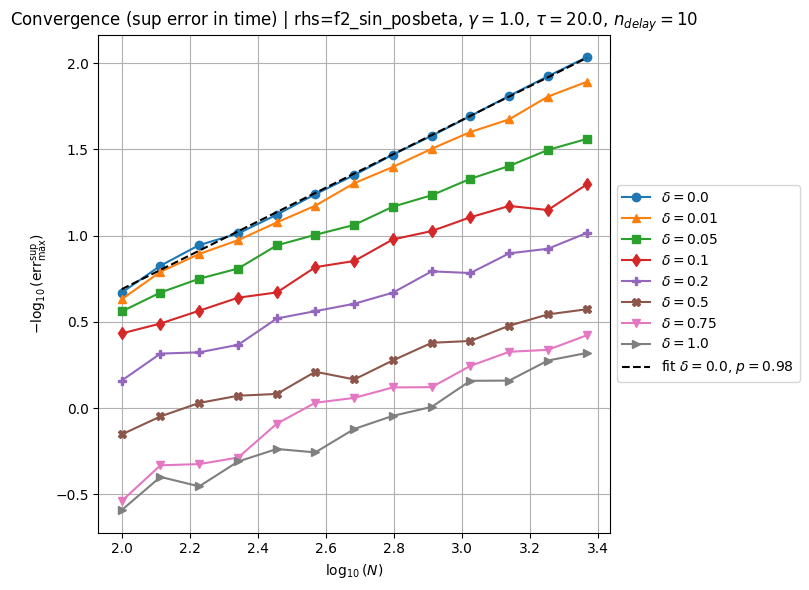}
        \caption{$f_2$, $\gamma = 1.0$}
        \label{fig:f2-gamma-1}
    \end{subfigure}\hfill
    \begin{subfigure}[t]{0.4\textwidth}
        \centering
        \includegraphics[width=\textwidth]{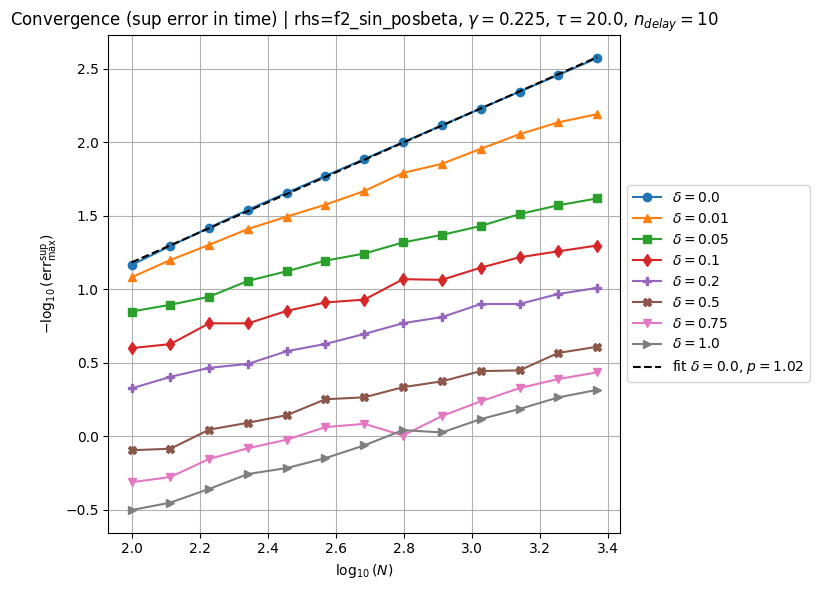}
        \caption{$f_2$, $\gamma = 0.225$}
        \label{fig:f2-gamma-0.225}
    \end{subfigure}

    \vspace{0.4cm}

    \begin{subfigure}[t]{0.4\textwidth}
        \centering
        \includegraphics[width=\textwidth]{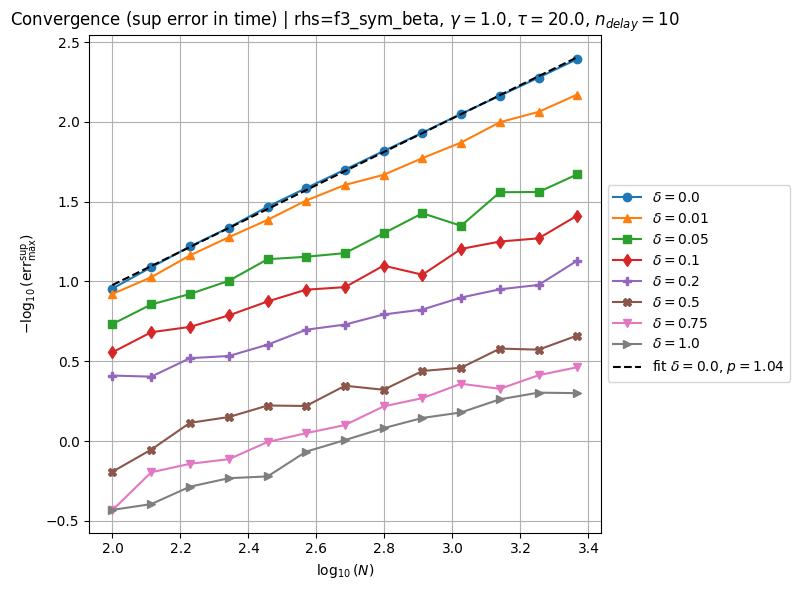}
        \caption{$f_3$, $\gamma = 1.0$}
        \label{fig:f3-gamma-1}
    \end{subfigure}\hfill
    \begin{subfigure}[t]{0.4\textwidth}
        \centering
        \includegraphics[width=\textwidth]{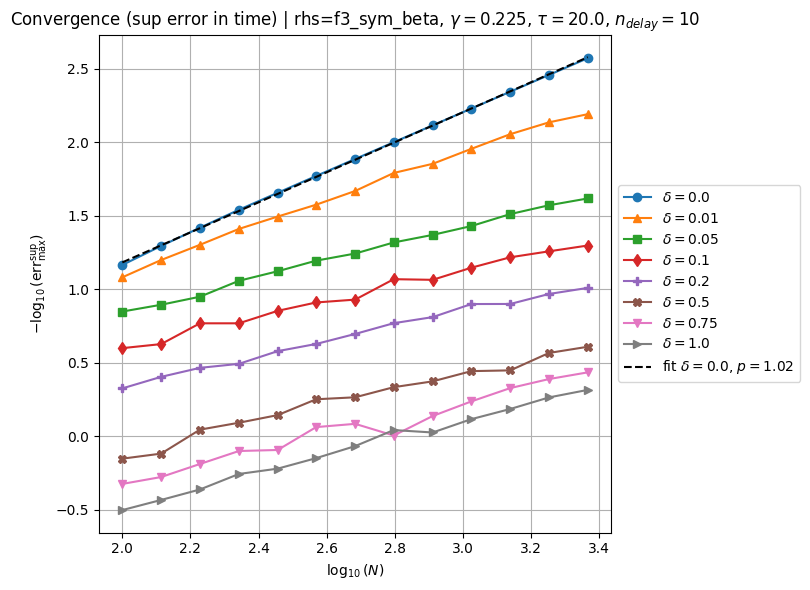}
        \caption{$f_3$, $\gamma = 0.225$}
        \label{fig:f3-gamma-0.225}
    \end{subfigure}

    \vspace{0.4cm}

    \begin{subfigure}[t]{0.4\textwidth}
        \centering
        \includegraphics[width=\textwidth]{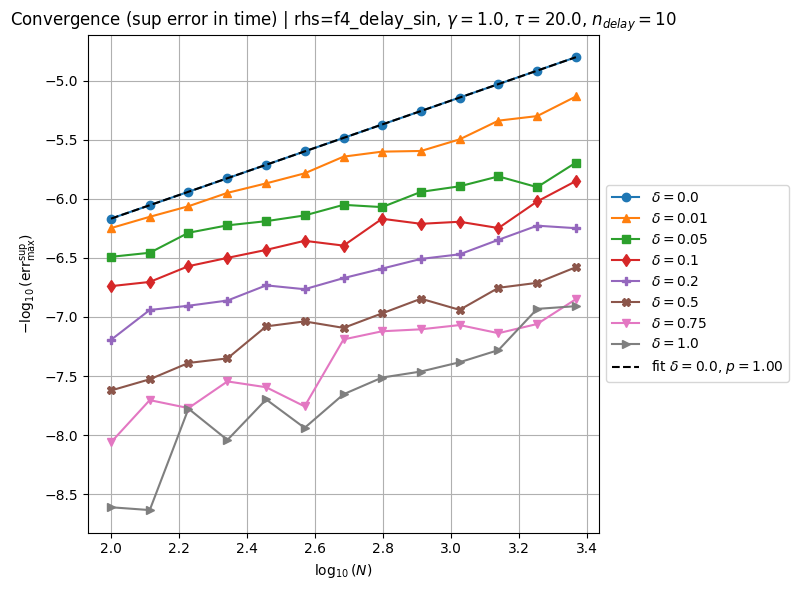}
        \caption{$f_4$, $\gamma = 1.0$}
        \label{fig:f4-gamma-1}
    \end{subfigure}\hfill
    \begin{subfigure}[t]{0.4\textwidth}
        \centering
        \includegraphics[width=\textwidth]{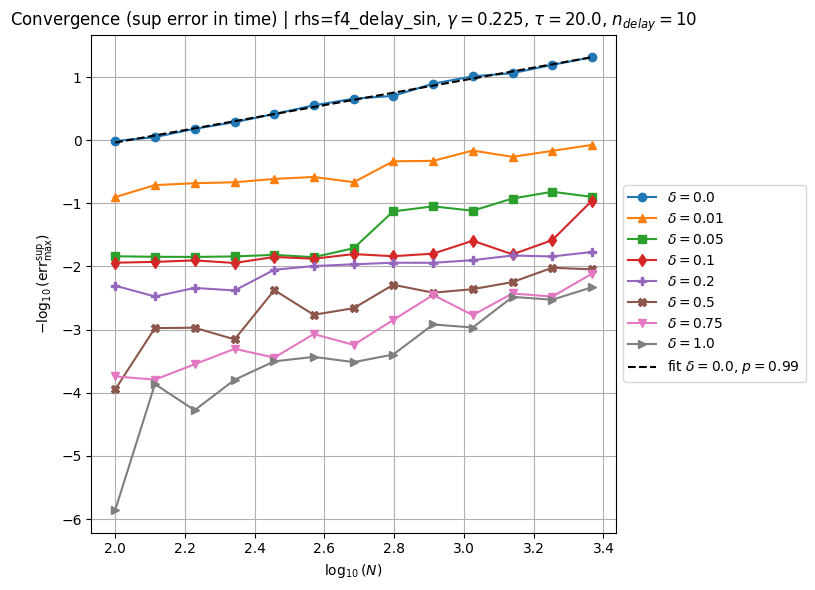}
        \caption{$f_4$, $\gamma = 0.225$}
        \label{fig:f4-gamma-0.225}
    \end{subfigure}

    \caption{Empirical convergence plots for the four test equations $f_1,f_2,f_3,f_4$ and two values of $\gamma$.}
    \label{fig:convergence-4x2}
\end{figure}

\subsection{Interval-wise and cumulative supremum errors}

To track error propagation across delay intervals, for \(j=0,\dots,n\) we define
\[
E^{\mathrm{loc}}_j
:=\sup_{t\in[j\tau,(j+1)\tau]}
|y_h^\delta(t)-y_{\mathrm{ref}}(t)|,
\]
\[
E^{\mathrm{cum}}_j
:=\sup_{t\in[0,(j+1)\tau]}
|y_h^\delta(t)-y_{\mathrm{ref}}(t)|.
\]
Here \(E^{\mathrm{loc}}_j\) measures the local error on the \(j\)-th interval, while \(E^{\mathrm{cum}}_j\) measures accumulated error up to time \((j+1)\tau\). By definition, \(\{E^{\mathrm{cum}}_j\}\) is nondecreasing.

The cumulative-supremum plots in Fig.~\ref{fig:cum-sup-4x2} show systematic growth with \(\delta\). For \(\gamma=1\), growth across intervals is milder. For \(\gamma<1\), accumulation is stronger, and reducing \(h\) becomes less effective over long horizons.

\begin{figure}[htbp]
    \centering

    \begin{subfigure}[t]{0.48\textwidth}
        \centering
        \includegraphics[width=\textwidth]{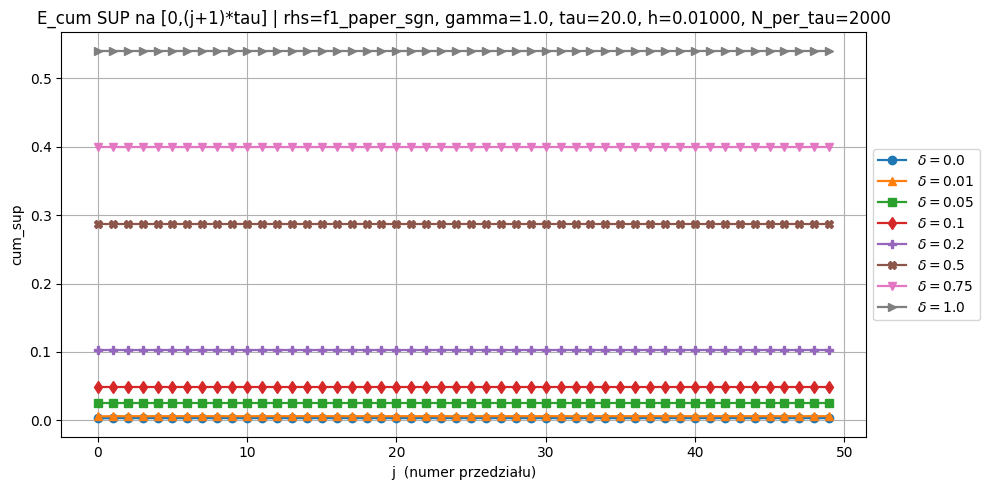}
        \caption{$f_1$, $\gamma = 1.0$}
        \label{fig:f1-gamma-1_cum_sup}
    \end{subfigure}\hfill
    \begin{subfigure}[t]{0.48\textwidth}
        \centering
        \includegraphics[width=\textwidth]{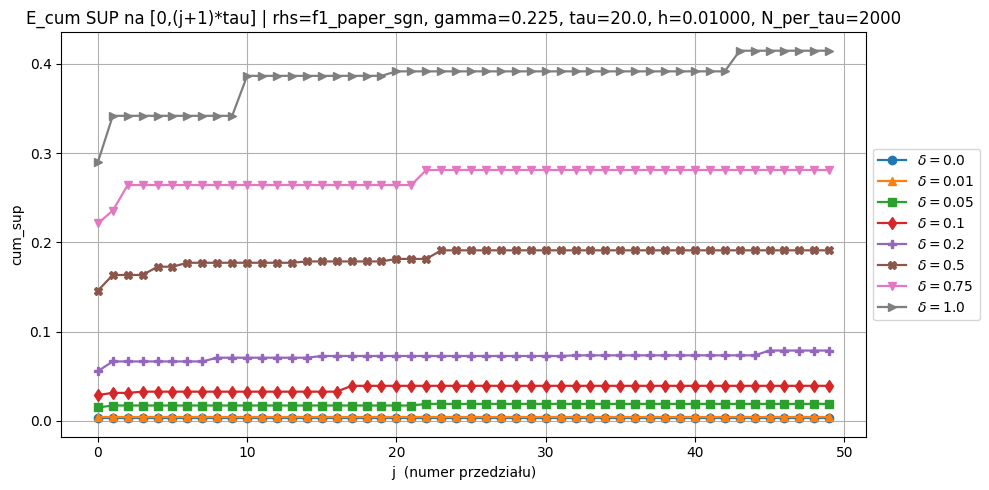}
        \caption{$f_1$, $\gamma = 0.225$}
        \label{fig:f1-gamma-0.225_cum_sup}
    \end{subfigure}

    \vspace{0.4cm}

    \begin{subfigure}[t]{0.48\textwidth}
        \centering
        \includegraphics[width=\textwidth]{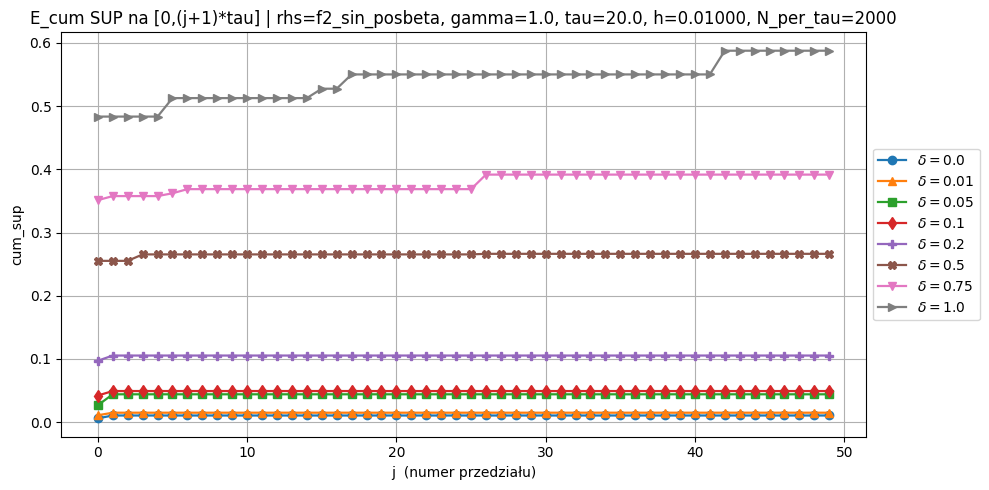}
        \caption{$f_2$, $\gamma = 1.0$}
        \label{fig:f2-gamma-1_cum_sup}
    \end{subfigure}\hfill
    \begin{subfigure}[t]{0.48\textwidth}
        \centering
        \includegraphics[width=\textwidth]{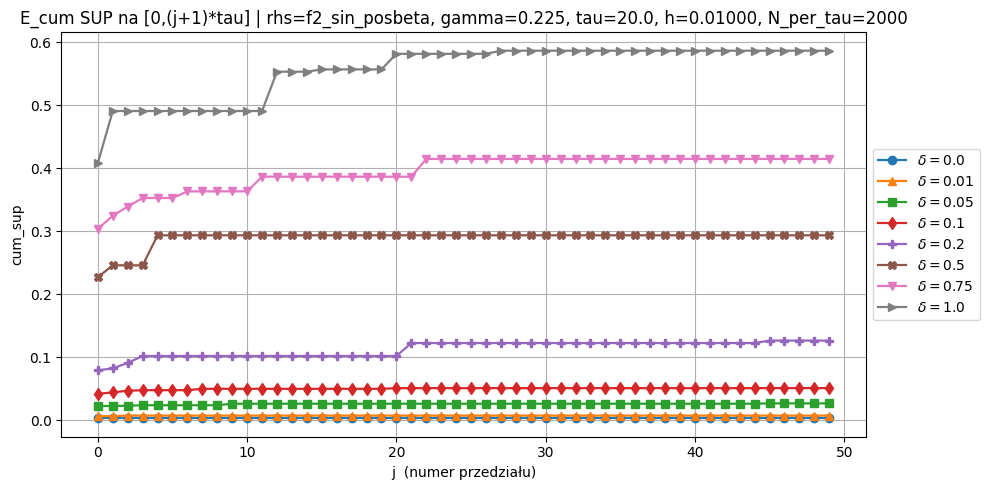}
        \caption{$f_2$, $\gamma = 0.225$}
        \label{fig:f2-gamma-0.225_cum_sup}
    \end{subfigure}

    \vspace{0.4cm}

    \begin{subfigure}[t]{0.48\textwidth}
        \centering
        \includegraphics[width=\textwidth]{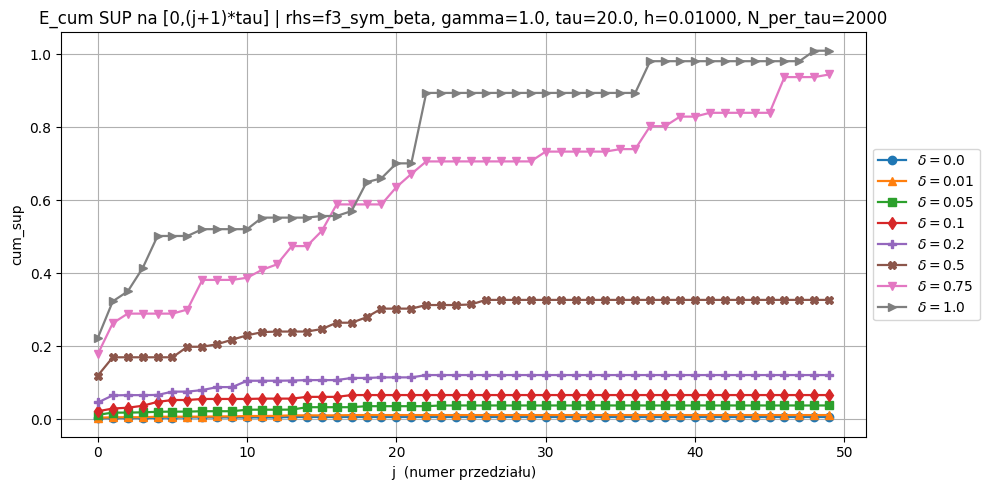}
        \caption{$f_3$, $\gamma = 1.0$}
        \label{fig:f3-gamma-1_cum_sup}
    \end{subfigure}\hfill
    \begin{subfigure}[t]{0.48\textwidth}
        \centering
        \includegraphics[width=\textwidth]{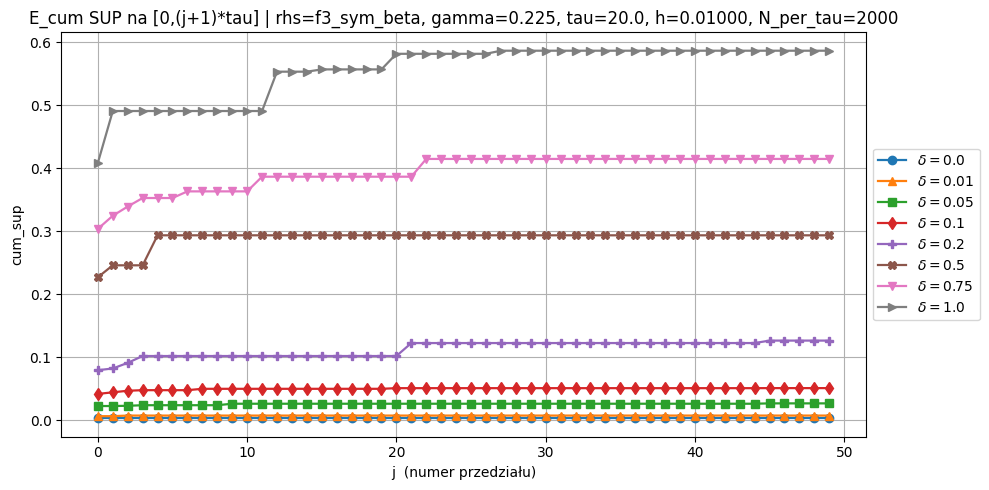}
        \caption{$f_3$, $\gamma = 0.225$}
        \label{fig:f3-gamma-0.225_cum_sup}
    \end{subfigure}

    \vspace{0.4cm}

    \begin{subfigure}[t]{0.48\textwidth}
        \centering
        \includegraphics[width=\textwidth]{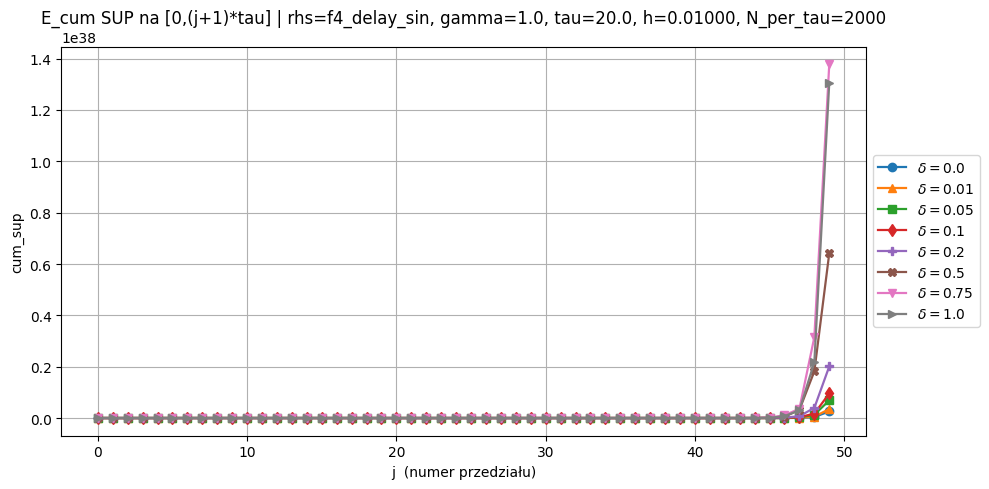}
        \caption{$f_4$, $\gamma = 1.0$}
        \label{fig:f4-gamma-1_cum_sup}
    \end{subfigure}\hfill
    \begin{subfigure}[t]{0.48\textwidth}
        \centering
        \includegraphics[width=\textwidth]{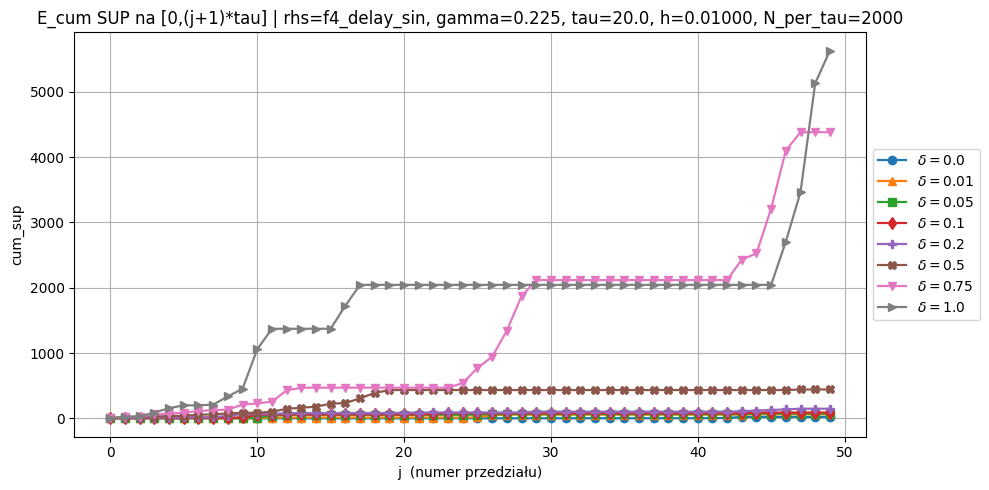}
        \caption{$f_4$, $\gamma = 0.225$}
        \label{fig:f4-gamma-0.225_cum_sup}
    \end{subfigure}

    \caption{Cumulative supremum errors for the four test equations \(f_1,f_2,f_3,f_4\) and two values of \(\gamma\).}
    \label{fig:cum-sup-4x2}
\end{figure}

\subsection{Summary}
The experiments support the theoretical picture. The method is stable under small perturbations, but the long-time effect of information noise depends strongly on \(\gamma\). In particular, the transition from \(\gamma=1\) to \(\gamma<1\) amplifies cumulative error over successive delay intervals.

\section{Conclusion and future work}
This paper establishes upper error bounds for Euler-type approximations of nonlinear delay differential equations under inexact information, covering both globally Lipschitz and non-globally Lipschitz regimes. Numerical experiments corroborate the predicted convergence rates. Future work will focus on deriving matching lower bounds (complexity lower limits) to quantify the sharpness of these rates, and on extending the analysis to adaptive methods and broader classes of perturbed-data problems.

\section*{Funding}
Funding not applicable.

\appendix
\section{}\label{appendixA}

\begin{lemma}[Gr\"onwall's Lemma - discrete version]\label{gronwall}
Let $A, B \geq 0$.  
If the sequence $\{\xi_n\}_{n=0}^{+\infty} \subset \mathbb{R}$ satisfies the inequality  
for every $n \in \mathbb{N}_0$
\[
|\xi_{n+1}| \leq A |\xi_n| + B,
\]
then for every $n \in \mathbb{N}$ we have
\[
|\xi_n| \leq A^n |\xi_0| + C_n,
\]
where 
\[
C_n = 
\begin{cases}
\displaystyle \frac{A^n - 1}{A - 1} B, & \text{if } A \neq 1, \\[1.2ex]
nB, & \text{if } A = 1.
\end{cases}
\]
\end{lemma}

\begin{lemma} \label{lem3.1} (\cite{ncz_pp_pm}) 
Let $\eta \in \mathbb{R}^d$ and let $f$ satisfy {\rm (F1)--(F4)}. Moreover, fix $\tau \in (0,+\infty)$ and 
$n \in \mathbb{Z}_+ \cup \{0\}$. Then the \eqref{main_eq} has a unique solution
\begin{equation}\label{3.1}
    z \in C^1([0,(n+1)\tau]; \mathbb{R}^d).
\end{equation}
Moreover, then there exist $K_0, K_1, \ldots, K_n \ge 0$ such that for $j = 0, 1, \ldots, n$
\begin{equation}\label{3.2}
    \sup_{j\tau \le t \le (j+1)\tau} \|\phi_j(t)\| \le K_j,
\end{equation}
and, for all $t, s \in [j\tau, (j+1)\tau]$
\begin{equation}\label{3.3}
    \|\phi_j(t) - \phi_j(s)\| \le \bar{K}_j |t - s|,
\end{equation}
with $\bar{K}_j = K(1 + K_{j-1})(1 + K_j)$, where $K_{-1} := \|\eta\|$.
\end{lemma}

\end{document}